\documentclass[12pt,reqno]{amsart}

\usepackage{graphicx}
\usepackage[top=20mm, bottom=20mm, left=30mm, right=30mm]{geometry}
\usepackage[colorlinks=true,citecolor=black,linkcolor=black,urlcolor=magenta]{hyperref}

\usepackage{fancyvrb}
\usepackage{ulem} 
\usepackage[T1]{fontenc} 
\usepackage{pbsi} 
\usepackage{url} 

\newcommand{\mathsmaller}[1]{#1}

\usepackage{local-math-commands}

\usepackage{enumitem}
\setlist[itemize]{leftmargin=0.35in}

\usepackage{diagbox}

\usepackage{ifthen}
\newcommand{\Hn}[2]{
     \ifthenelse{\equal{#2}{1}}{H_{#1}}{H_{#1}^{\left(#2\right)}}
}

\DeclareMathOperator{\ab}{ab}
\DeclareMathOperator{\ConvP}{P}
\DeclareMathOperator{\ConvQ}{Q}

\usepackage{framed} 
\usepackage{caption} 
\title[Jacobi-Type Continued Fractions and Congruences for Binomial Coefficients]{
       Jacobi-Type Continued Fractions and Congruences for 
       Binomial Coefficients Modulo Integers $h \geq 2$} 

\author{Maxie D. Schmidt \\ 
        \href{mailto:maxieds@gmail.com}{maxieds@gmail.com}}         
\date{}
\address{University of Washington \\ 
        Department of Mathematics \\ 
        Padelford Hall \\ 
        Seattle, WA 98195 \\ 
        USA 
} 

\allowdisplaybreaks

\begin{document}

\begin{abstract} 
We prove two new forms of Jacobi-type J-fraction expansions generating the 
binomial coefficients, $\binom{x+n}{n}$ and $\binom{x}{n}$, over all $n \geq 0$. 
Within the article we establish new forms of integer congruences for these 
binomial coefficient variations modulo any (prime or composite) $h \geq 2$ and 
compare our results with existing known congruences for the binomial coefficients 
modulo primes $p$ and prime powers $p^k$. 
We also prove new exact formulas for these binomial coefficient cases from the 
expansions of the $h^{th}$ convergent functions to the infinite J-fraction 
series generating these coefficients for all $n$. 
\end{abstract} 

\maketitle

\section{Introduction} 

\subsection{Congruences for Binomial Coefficients Modulo Primes and Prime Powers} 
\label{subSection_KnownResults_BinomCoeffs} 

There are many well-known results providing congruences for the 
binomial coefficients modulo primes and prime powers. 
For example, we can state \emph{Lucas's theorem} in the following form 
for $p$ prime and $n, m \in \mathbb{N}$ where 
$n = n_0+n_1 p+\cdots+n_d p^d$ and $m = m_0+m_1 p+\cdots+m_d p^d$ for 
$0 \leq n_i, m_i < p$ \citep{GRANVILLE}: 
\begin{align*}
\tag{Lucas' Theorem} 
\binom{n}{m} & \equiv \binom{n_0}{m_0} \binom{n_1}{m_1} \cdots \binom{n_d}{m_d} && 
     \pmod{p}. 
\end{align*} 
We also have known results providing decompositions, or reductions of order, of the 
next binomial coefficients modulo the prime powers, $p^k$ \citep{MESTROVIC}. 
\begin{align*} 
\binom{n p^k+n_0}{m p^k+m_0} & \equiv \binom{np}{mp} \binom{n_0}{m_0} && 
     \pmod{p^k} 
\end{align*} 
Similarly, we can decompose linear and quadratic (and more general) polynomial 
inputs to the lower binomial coefficient indices as the congruence factors given by 
\citep[\S 3.2]{MESTROVIC} 
\begin{align*} 
\binom{np}{rp+s} & \equiv (r+1) \binom{n}{r+1} \binom{p}{s} && \pmod{p^2} \\ 
\binom{np}{mp^2+rp+s} & \equiv (m+1) \binom{n}{m+1} \binom{p^2}{rp+s} 
     && \pmod{p^3}. 
\end{align*} 
Within the context of this article, 
we are motivated to establish meaningful, non-trivial integer congruences for 
the two indeterminate binomial coefficient sequence variants, 
$\binom{x+n}{n}$ and $\binom{x}{n}$, modulo both prime and composite bases 
$h \geq 2$ by using 
new expansions of the Jacobi-type continued fractions generating these 
binomial coefficient variants that we prove in 
Section \ref{Section_BinomCoeffs_xpnn} and in Section \ref{Section_BinomCoeffs_xn}. 
The next subsection establishes related known continued fractions and 
integer congruence properties for the general forms of Jacobi-type J-fractions 
of which our new results are special cases. 

\subsection{Jacobi-Type Continued Fraction Expansions} 
\label{subSection_Intro_J-fraction_exps} 

\subsubsection*{Continued Fractions Generating the Binomial Coefficients} 

In this article, 
we study new properties and congruence relations satisfied by the 
integer-order binomial coefficients through two specific, and 
apparently new, \emph{Jacobi-type continued fraction} expansions of a 
formal power series in $z$. The expansions of these \emph{J-fractions} 
are similar in form 
to a known \emph{Stieltjes-type} continued fraction, or \emph{S-fraction}, 
expansion given in Wall's book as 
\citep[p.\ 343]{WALL-CFRACS} 
\begin{align*} 
\tag{Binomial Series S-Fraction} 
(1+z)^k & = \cfrac{1}{
                 1-\cfrac{kz}{
                      1 + \cfrac{\frac{1 \cdot (1+k)}{1 \cdot 2} z}{
                           1 + \cfrac{\frac{1 \cdot (1-k)}{2 \cdot 3} z}{
                                1 + \cfrac{\frac{2 (2+k)}{3 \cdot 4} z}{ 
                                     1 + \cdots. 
                                } 
                           }
                      }
                 }
            }
\end{align*} 

\subsubsection*{Jacobi-Type J-Fractions for Generalized Factorial Functions} 

The loosely-termed ``\emph{non-exponential}'' forms of a generalized class of 
binomial-coefficient-related \emph{generalized factorial functions}, 
$p_n(\alpha, R)$, defined as 
\begin{align*} 
p_n(\alpha, R) & := 
     R(R+\alpha)(R+2\alpha) \cdots (R + (n-1)\alpha), 
\end{align*} 
where the special cases of $n! = p_n(-1, n)$, 
$\binom{x}{n} \equiv (-1)^n (-x)_n / n!$, and 
$p_n(\alpha, R) / n! \equiv (-\alpha)^n \cdot \binom{-R / \alpha}{n}$ when 
$(x)_n = x(x+1)\cdots(x+n-1)$ denotes the \emph{Pochhammer symbol} 
are studied in the reference \citep{MULTIFACT-CFRACS}. 
Exponential generating functions for the generalized symbolic product 
functions, $p_n(\alpha, R)$, with respect to $n$ are known and 
given in closed-form by 
\begin{align*} 
\sum_{n \geq 0} p_n(\alpha, R+1) \frac{z^n}{n!} & = 
     \left(1 - \alpha z\right)^{-(R+1) / \alpha}, 
\end{align*} 
where the corresponding generalized forms of the 
\emph{Stirling numbers of the first kind}, or 
\emph{$\alpha$-factorial coefficients}, are defined by the 
generating functions \citep[\cf \S 2-3]{MULTIFACTJIS} 
\begin{align*} 
\sum_{n \geq 0} [R^m] p_n(\alpha, R+1) \frac{z^n}{n!} & = 
     \frac{(-1)^m \alpha^{-m}}{m!} \times (1-\alpha z)^{\frac{1}{\alpha}} 
     \times \Log(1-\alpha z)^m,\ \text{ for } m \in \mathbb{Z}^{+}. 
\end{align*} 
The power series expansions of the ordinary generating functions for these 
factorial functions, which typically converge only for $z = 0$ when the 
parameter $R$ is a function of $n$, are defined formally in 
\citep{MULTIFACT-CFRACS} by infinite J-fractions of the form 
\begin{align*} 
\ConvGF{\infty}{\alpha}{R}{z} & \phantom{:} = 
     \cfrac{1}{1 - R \cdot z - 
     \cfrac{\alpha R \cdot z^2}{ 
            1 - (R+2\alpha) \cdot z -
     \cfrac{2\alpha (R + \alpha) \cdot z^2}{ 
            1 - (R + 4\alpha) \cdot z - 
     \cfrac{3\alpha (R + 2\alpha) \cdot z^2}{ 
     \cdots.}}}} 
\end{align*} 
These typically divergent ordinary generating functions are 
usually only expanded as 
power series in $z$ by ``\emph{regularized}'' \emph{Borel} sums 
involving reciprocal powers of $z$ (and $\alpha z$) such as those 
given as examples in the introduction to \citep{MULTIFACT-CFRACS}. 

\subsubsection*{Properties of the Expansions of General J-Fraction Series} 

More generally, Jacobi-type \emph{J-fractions} correspond to power series defined by 
infinite continued fraction expansions of the form 
\begin{align*} 
\tag{J-Fraction Expansions} 
J^{[\infty]}(z) & = 
     \cfrac{1}{1-c_1z-\cfrac{\ab_1 z^2}{1-c_2z- 
     \cfrac{\ab_2 z^2}{\cdots},}}
\end{align*} 
for some sequences $\{ c_i \}_{i \geq 1}$ and 
$\{ \ab_i \}_{i \geq 1}$, and 
some (typically formal) series variable $z \in \mathbb{C}$ 
\citep[\cf \S 3.10]{NISTHB} \citep{WALL-CFRACS,FLAJOLET80B,FLAJOLET82}. 
The formal series enumerated by special cases of the truncated and infinite 
continued fraction series of this form include \emph{ordinary} 
(as opposed to typically closed-form \emph{exponential}) 
\emph{generating functions} 
for many one and two-index combinatorial sequences including the 
generalized factorial functions studied in the references 
\citep{FLAJOLET80B,FLAJOLET82,GFLECT,MULTIFACT-CFRACS}. 

The $h^{th}$ convergents, $\Conv_h(z)$, 
to the J-fraction expansions of the form in the 
previous equation have several useful characteristic properties: 
\begin{itemize} 

\item[(A)] 
The $h^{th}$ convergent functions exactly enumerate the coefficients of the 
infinite continued fraction series as 
$[z^n] J^{[\infty]}(z) = [z^n] \Conv_h(z)$ for all $0 \leq n < 2h$. 

\item[(B)] 
The convergent function component numerator and denominator sequences 
are each easily defined recursively by the same second-order 
recurrence in $h$ with differing initial conditions. 
The $h^{th}$ convergent functions are always rational in $z$ 
which implies an ``\emph{eventually periodic}'' nature to their 
coefficients, as well as $h$-order finite difference equations satisfied by the 
coefficients of these truncated series approximations. 

\item[(C)] 
If $M_h := \ab_1 \cdots \ab_h$ denotes the \emph{$h^{th}$ modulus} of the 
generalized J-fraction expansion defined above, the 
sequence of coefficients enumerated by the 
$h^{th}$ convergent function is eventually periodic modulo $M_h$ and 
satisfies a finite difference equation of order at most $h$ 
\citep[Thm.\ 1]{FLAJOLET82}. 
Moreover, 
if $M_m$ is integer-valued and $h \geq 2$ divides $M_m$ for some $m \geq h$, then 
$[z^n] J^{[\infty]} \equiv [z^n] \Conv_m(z) \pmod{h}$ \citep[\S 5.7]{GFLECT}. 

\end{itemize} 
Typically we are only interested in the congruences formed by the $h^{th}$ 
convergent functions for the coefficients of $z^n$ the infinite series, 
$J^{[\infty]}(z)$ in $z$ for $0 \leq n \leq h$, 
which are in fact exactly enumerated by these 
convergent functions. Other properties of the convergent 
functions, such as finite difference equations satisfied by their 
coefficients, suggest other non-obvious and non-trivial properties 
of the resulting congruences guaranteed by the convergent functions. 

\subsection{Organization of the Article} 

We provide two new forms of infinite J-fractions generating the binomial 
coefficients, $\binom{x+n}{n}$ and $\binom{x}{n}$, for $n \geq 0$ and any 
non-zero indeterminate $x$ in the next sections of the article. 
These new continued fraction results lead to new exact formulas and 
finite difference equations for these 
binomial coefficient variants, and new congruences for the binomial 
coefficients modulo \emph{any} (prime or composite) integers $h \geq 2$ whenever 
$\frac{1}{2} \binom{x+h-1}{h-1} \binom{x}{h-1} / \binom{2h-3}{h-2}^{2} 
 \in \mathbb{Z}$ and 
$h | \frac{1}{2} \binom{x+h-1}{h-1} \binom{x}{h-1} / \binom{2h-3}{h-2}^{2}$. 
Addition formulas for these J-fractions imply new 
identities for reductions of order of the upper index, $x$, comparable to the 
statements of Lucas' theorem and the other results modulo primes $p$ and 
prime powers $p^k$ given in Section \ref{subSection_KnownResults_BinomCoeffs}. 

The proofs of these new continued fraction expansions mostly follow by 
inductive arguments applied to known recurrence relations for the $h^{th}$ 
numerator and denominator convergent function sequences. 
The proofs that these infinite J-fraction expansions exactly enumerate our 
binomial coefficient variants of interest follow from the rationality of the 
$h^{th}$ convergent functions in $z$ for all $h \geq 1$. 
Consequences of the proofs of our main theorems include new exact formulas for the 
binomial coefficients, $\binom{x+n}{n}$ and $\binom{x}{n}$, and new 
congruence properties for these binomial coefficient forms. 
Since the proofs of the corresponding results in each case given in 
Section \ref{Section_BinomCoeffs_xpnn} and Section \ref{Section_BinomCoeffs_xn}, 
respectively, are so similar, we give 
careful proofs of the results for the case of the J-fractions generating 
$\binom{x+n}{n}$ in the first section, and 
choose to only state the analogous results for the case of $\binom{x}{n}$ 
in the second section below for clarity of exposition. 

\section{J-Fraction Expansions for the Binomial Coefficients, $\binom{x+n}{n}$} 
\label{Section_BinomCoeffs_xpnn} 

\begin{definition}[Component Sequences and Convergent Functions for $\binom{x+n}{n}$] 
\label{def_ComponentSeqs_ConvFns_BinomCoeff_xpnn} 
For a non-zero indeterminate $x$, let the sequences, 
$c_{x,i}^{(1)}$ and $\ab_{x,i}^{(1)}$, 
be defined over $i \geq 1$ as follows: 
\begin{align*} 
c_{x,i}^{(1)} & := -\frac{1}{(2i-1)(2i-3)}\left(1+2(i-2)i-x\right) \\ 
\ab_{x,i}^{(1)} & := 
     \begin{cases} 
     -\frac{1}{4 (2i-3)^2} (x-i+2)(x+i-1) & \text{ if } i \geq 3 \\ 
     -\frac{1}{2}x(x+1) & \text{ if } i = 2 \\ 
     0 & \text{ otherwise. }
     \end{cases} 
\end{align*} 
We then define the $h^{th}$ convergent functions, 
$\Conv_{1,h}(x, z) := \ConvP_{1,h}(x, z) / \ConvQ_{1,h}(x, z)$, 
through the component numerator and denominator functions given recursively by 
\begin{align} 
\label{eqn_ConvFns_BinomCoeff_xpnn} 
\ConvP_{1,h}(x, z) & = (1-c_{x,h}^{(1)}z) \ConvP_{1,h-1}(x, z) - 
     \ab_{x,h}^{(1)} z^2 \ConvP_{1,h-1}(x, z) + \Iverson{h = 1} \\ 
\notag 
\ConvQ_{1,h}(x, z) & = (1-c_{x,h}^{(1)}z) \ConvQ_{1,h-1}(x, z) - 
     \ab_{x,h}^{(1)} z^2 \ConvQ_{1,h-1}(x, z) + 
     (1-c_{x,1}^{(1)}z) \Iverson{h = 1} + \Iverson{h = 0}. 
\end{align} 
Listings of the first several cases of the numerator and denominator 
convergent functions, $\ConvP_{1,h}(x, z)$ and $\ConvQ_{1,h}(x, z)$, 
are given in Table \ref{table_ConvP2hxz_SpCases} and 
Table \ref{table_ConvQ2hxz_SpCases}, respectively. 
\end{definition} 

\begin{table} 
\begin{framed} 
\centering\small

\begin{tabular}{||c|l||} \hline\hline 
$h$ & $\ConvP_{1,h}(x, z)$ \\ \hline
0 & $0$ \\ 
1 & $1$ \\
2 & $1+\frac{1-x}{3}z$ \\
3 & $1-\frac{2}{5} (x-2) z+\frac{1}{20} (x-2) (x-1)z^2$ \\
4 & $1-\frac{3}{7} (x-3) z+\frac{1}{14} (x-3) (x-2) z^2- 
     \frac{1}{210} (x-3) (x-2) (x-1)z^3$ \\
5 & $1-\frac{4}{9} (x-4) z+\frac{1}{12} (x-4) (x-3) z^2- 
     \frac{1}{126} (x-4) (x-3) (x-2) z^3 +\frac{(x-4) (x-3) (x-2) (x-1)}{3024}z^4$ \\
6 & $1-\frac{5}{11} (x-5) z+\frac{1}{11} (x-5) (x-4) z^2- 
     \frac{1}{99} (x-5) (x-4) (x-3) z^3+ 
     \frac{(x-5) (x-4) (x-3) (x-2)}{1584} z^4$ \\ 
  & $\phantom{1}-
     \frac{(x-5) (x-4) (x-3) (x-2) (x-1)}{55440} z^5$ \\
\hline\hline 
\end{tabular} 

\bigskip 

\begin{tabular}{||c|l||} \hline\hline 
$h$ & $(2h-1)! \cdot \ConvP_{1,h}(x, z)$ \\ \hline
0 & $0$ \\ 
1 & $1$ \\
2 & $6-2 (x-1)z$ \\
3 & $120 -48 (x-2) z^2+6 (x-2) (x-1) z^2$ \\
4 & $5040-2160 (x-3) z+360 (x-3) (x-2) z^2-24 (x-3) (x-2) (x-1)z^3$ \\
5 & $362880-161280 (x-4) z+30240 (x-4) (x-3) z^2- 
     2880 (x-4) (x-3) (x-2) z^3$ \\ 
  & $\phantom{362880} + 
     120 (x-4) (x-3) (x-2) (x-1)z^4$ \\
6 & $39916800-18144000 (x-5) z+3628800 (x-5) (x-4) z^2- 
     403200 (x-5) (x-4) (x-3) z^3$ \\ 
  & $\phantom{39916800} + 
     25200 (x-5) (x-4) (x-3) (x-2) z^4- 
     720 (x-5) (x-4) (x-3) (x-2) (x-1) z^5$ \\ 
\hline\hline 
\end{tabular} 

\bigskip 

\captionof{table}{The Numerator Convergent Functions, $\ConvP_{1,h}(x, z)$, 
                  Generating the Binomial Coefficients, $\binom{x+n}{n}$} 
\label{table_ConvP2hxz_SpCases} 

\end{framed} 
\end{table} 

\begin{table}[ht] 
\begin{framed} 
\centering\small

\begin{tabular}{||c|l||} \hline\hline 
$h$ & $\ConvQ_{1,h}(x, z)$ \\ \hline
1 & $1 + (x+1) z$ \\
2 & $1+\frac{2}{3} (x+2) z+\frac{1}{6} (x+1) (x+2)z^2$ \\
3 & $1+\frac{3}{5} (x+3) z+\frac{3}{20} (x+2) (x+3) z^2+ 
     \frac{1}{60} (x+1) (x+2) (x+3)z^3$ \\
4 & $1+\frac{4}{7} (x+4) z+\frac{1}{7} (x+3) (x+4) z^2+ 
     \frac{2}{105} (x+2) (x+3) (x+4) z^3$ \\ 
  & $\phantom{1} + 
     \frac{1}{840} (x+1) (x+2) (x+3) (x+4)z^4$ \\
5 & $1+\frac{5}{9} (x+5) z+\frac{5}{36} (x+4) (x+5) z^2+ 
     \frac{5}{252} (x+3) (x+4) (x+5) z^3+ 
     \frac{5 (x+2) (x+3) (x+4) (x+5)}{3024} z^4$ \\ 
  & $\phantom{1} + 
     \frac{(x+1) (x+2) (x+3) (x+4) (x+5)}{15120} z^5$ \\
6 & $1+\frac{6}{11} (x+6) z+\frac{3}{22} (x+5) (x+6) z^2+ 
     \frac{2}{99} (x+4) (x+5) (x+6) z^3$ \\ 
  & $\phantom{1} + 
     \frac{1}{528} (x+3) (x+4) (x+5) (x+6) z^4+ 
     \frac{(x+2) (x+3) (x+4) (x+5) (x+6)}{9240} z^5$ \\ 
  & $\phantom{1} + 
     \frac{(x+1) (x+2) (x+3) (x+4) (x+5) (x+6)}{332640} z^6$ \\
\hline\hline 
\end{tabular} 

\bigskip 

\begin{tabular}{||c|l||} \hline\hline 
$h$ & $(2h-1)! \cdot \ConvQ_{1,h}(x, z)$ \\ \hline
1 & $1+(x+1)z$ \\
2 & $6+4 (x+2) z+(x+1) (x+2)z^2$ \\
3 & $120+72 (x+3) z+18 (x+2) (x+3) z^2+2 (x+1) (x+2) (x+3)z^3$ \\
4 & $5040+2880 (x+4) z+720 (x+3) (x+4) z^2+ 
     96 (x+2) (x+3) (x+4) z^3$ \\ 
  & $\phantom{5040} + 
     6 (x+1) (x+2) (x+3) (x+4) z^4$ \\
5 & $362880+201600 (x+5) z+50400 (x+4) (x+5) z^2+ 
     7200 (x+3) (x+4) (x+5) z^3$ \\ 
  & $\phantom{3622880} + 
     600 (x+2) (x+3) (x+4) (x+5) z^4+ 
     24 (x+1) (x+2) (x+3) (x+4) (x+5) z^5$ \\
6 & $39916800+21772800 (x+6) z+5443200 (x+5) (x+6) z^2+ 
     806400 (x+4) (x+5) (x+6) z^3$ \\ 
  & $\phantom{39916800} + 
     75600 (x+3) (x+4) (x+5) (x+6) z^4+ 
     4320 (x+2) (x+3) (x+4) (x+5) (x+6) z^5$ \\ 
  & $\phantom{39916800} + 
     120 (x+1) (x+2) (x+3) (x+4) (x+5) (x+6) z^6$ \\
\hline\hline 
\end{tabular} 

\bigskip 

\captionof{table}{The Denominator Convergent Functions, $\ConvQ_{1,h}(x, z)$, 
                  Generating the Binomial Coefficients, $\binom{x+n}{n}$} 
\label{table_ConvQ2hxz_SpCases} 

\end{framed} 
\end{table} 

\begin{prop}[Formulas for the Numerator Convergent Functions] 
\label{prop_NumConvFn_Formula}
For all $h \geq 1$ and indeterminate $x$, the 
numerator convergent functions, $\ConvP_{1,h}(x, z)$, 
have the following formulas: 
\begin{align*} 
\ConvP_h(x, z) & = 
     \sum_{n=0}^{h-1} \binom{x+n-h}{n} 
     \frac{(h-1)!}{\left(h-1-n\right)!} \cdot 
     \frac{(2h-1-n)!}{(2h-1)!} \cdot (-z)^n \\ 
     & = 
     \sum_{n=0}^{h-1} \binom{x+n-h}{n} \binom{h-1}{n} \binom{2h-1}{n}^{-1} 
     \cdot (-z)^n. 
\end{align*} 
\end{prop} 
\begin{proof} 
The proof follows from \eqref{eqn_ConvFns_BinomCoeff_xpnn} of 
Definition \ref{def_ComponentSeqs_ConvFns_BinomCoeff_xpnn} by induction on $h$. 
The claimed formula holds for $h = 0, 1, 2$ by the computations given in 
Table \ref{table_ConvP2hxz_SpCases}. 
Suppose that $h \geq 3$ and that the two equivalent formulas stated in the 
proposition for $\ConvP_{1,k}(x, z)$ are correct for all $k < h$. 
In particular, the stated formula holds when $k = h-1, h-2$. 
Then by expanding \eqref{eqn_ConvFns_BinomCoeff_xpnn} using our hypothesis, 
we have that 
\begin{align*} 
\ConvP_{1,k}(x, z) & = \left(1+\frac{(1+2h(h-2)-x)z}{(2h-1)(2h-3)}\right) \times 
     \sum_{n=0}^{h-2} \binom{x+n+1-h}{n} \binom{h-2}{n} \binom{2h-3}{n}^{-1} 
     (-z)^n \\ 
     & \phantom{=\ } + 
     \frac{(x+2-h)(x+h-1) z^2}{4 (2h-3)^2} \times 
     \sum_{n=0}^{h-3} \binom{x+n+2-h}{n} \binom{h-3}{n} \binom{2h-5}{n}^{-1} 
     (-z)^n \\ 
     & = 
     \sum_{n=0}^{h-1} \binom{x+n+1}{n} \binom{h-2}{n} \binom{2h-3}{n}^{-1} (-z)^n \\ 
     & \phantom{=\ } - 
     \frac{(1+2h(h-2)-x)}{(2h-1)(2h-3)} \sum_{n=1}^{h-1} 
     \binom{x+n-h}{n-1} \binom{h-2}{n-1} \binom{2h-3}{n-1}^{-1} (-z)^n \\ 
     & \phantom{=\ } + 
     \frac{(x+2-h)(x+h-1)}{4 (2h-3)^2} \sum_{n=2}^{h-1} 
     \binom{x+n-h}{n-2} \binom{h-3}{n-2} \binom{2h-5}{n-2}^{-1} (-z)^n \\ 
     & = 
     1 + \left(\binom{x+2-h}{1} \binom{h-2}{1} \binom{2h-3}{1}^{-1} - 
     \frac{(1+2h(h-2)-x)}{(2h-1)(2h-3)} \right) (-z) \\ 
     & \phantom{=\ } + 
     \Biggl(\frac{(x+2-h)(x+h-1)}{4 (2h-3)^2} 
     \binom{x-1}{h-3} \binom{h-3}{h-3} \binom{2h-5}{h-3}^{-1} \\ 
     & \phantom{=+\Biggl(\ } - 
     \binom{x-1}{h-2} \binom{h-2}{h-2} \binom{2h-3}{h-2}^{-1} 
     \Biggr) (-z)^{h-1} \\ 
     & \phantom{=\ } + 
     \sum_{i=2}^{h-2} \binom{x+n-h}{n} \binom{h-1}{n} \binom{2h-1}{n}^{-1} 
     \Biggl(
     \frac{2(2h-1)(h-n-1)(x+n+1-h)}{(2h-2-n)(2h-1-n)(x+1-h)} \\ 
     & \phantom{=+\sum\ } - 
     \frac{2n(1+2h(h-2)-x)}{(2h-3)(2h-1-n)(x+1-h)} \\ 
     & \phantom{=+\sum\ } + 
     \frac{n(n-1)(2h-1)(x+2-h)(x-1+h)}{(2h-3)(2h-2-n)(2h-1-n)(x+2-h)(x+1-h)} 
     \Biggr) (-z)^n \\ 
     & = 
     1+\binom{x+1-h}{1} \binom{h-1}{1} \binom{2h-1}{1}^{-1} (-z) + 
     \binom{x-1}{h-1} \binom{h-1}{h-1} \binom{2h-1}{h-1}^{-1} (-z)^{h-1} \\ 
     & \phantom{=1\ } + 
     \sum_{n=2}^{h-2} \binom{x+n-h}{n} \binom{h-1}{n} \binom{2h-1}{n}^{-1} (-z)^n. 
\end{align*} 
We also notice the particular identity of use in proving 
Theorem \ref{theorem_MainThmI} below that the coefficients of the powers of $z^n$ 
on the right-hand-sides of the stated formulas satisfy 
\begin{align} 
\label{eqn_PropPhxz_alt_stmt} 
 & \mathsmaller{(-1)^n \binom{x+n-h}{n} \binom{h-1}{n} \binom{2h-1}{n}^{-1}} \\ 
\notag 
     & \phantom{(-1)^n\ } = 
     \sum_{i=0}^{n} \binom{x+n-i}{n-i} \binom{x+h}{i} \frac{h!}{(h-i)!} 
     \frac{(2h-1-i)!}{(2h-1)!} (-1)^{n-i}, 
\end{align} 
which is proved by summing exactly with \emph{Mathematica}. 
\end{proof} 

\begin{prop}[Formulas for the Denominator Convergent Functions] 
\label{prop_DenomConvFn_Formula}
For all $h \geq 0$ and fixed indeterminates $x$, we have that the 
denominator convergent functions, $\ConvQ_{1,h}(x, z)$, 
satisfy the following formula: 
\begin{align*} 
\ConvQ_h(x, z) & = \sum_{i=0}^{h} \binom{x+h}{i} \cdot \frac{h!}{(h-i)!} \cdot 
     \frac{(2h-1-i)!}{(2h-1)!} \cdot z^i. 
\end{align*} 
\end{prop} 
\begin{proof} 
The proof again follows from \eqref{eqn_ConvFns_BinomCoeff_xpnn} of 
Definition \ref{def_ComponentSeqs_ConvFns_BinomCoeff_xpnn} by induction on $h$. 
The claimed formula holds when $h = 0, 1, 2$ by the computations given in 
Table \ref{table_ConvQ2hxz_SpCases}. 
Next, we suppose that $h \geq 3$ and that the formula stated in the 
proposition for $\ConvQ_{1,k}(x, z)$ is correct for all $k < h$. 
In particular, the stated formulas hold when $k = h-1, h-2$. 
Then by expanding \eqref{eqn_ConvFns_BinomCoeff_xpnn} using our 
inductive hypothesis, we have that 
\begin{align*} 
\ConvQ_{1,h}(x, z) & = 
     \sum_{i=0}^{h-1} \binom{x+h-1}{i} \frac{(h-1)!}{(h-1-i)!} 
     \frac{(2h-3-i)!}{(2h-3)!} z^i \\ 
     & \phantom{=\ } + 
     \frac{(1+2h(h-2)-x)}{(2h-1)(2h-3)} \sum_{i=1}^{h} 
     \binom{x+h-1}{i-1} \frac{(h-1)!}{(h-i)!} \frac{(2h-2-i)!}{(2h-3)!} z^i \\ 
     & \phantom{=\ } + 
     \frac{(x+1-h)(x+h-1)}{4 (2h-3)^2} \sum_{i=2}^{h} 
     \binom{x+h-2}{i-2} \frac{(h-2)!}{(h-i)!} \frac{(2h-3-i)!}{(2h-5)!} z^i \\ 
     & = 
     1 + \binom{x+h}{1} \frac{h z}{(2h-1)} + 
     \binom{x+h}{h} \frac{h! (h-1)! z^h}{(2h-1)!} \\ 
     & \phantom{=1\ } + 
     \sum_{i=2}^{h-1} \binom{x+h}{i} \frac{h!}{(h-i)!} \frac{(2h-1-i)!}{(2h-1)!} z^i 
     \Biggl( 
     \frac{2(h-1)(2h-1)(h-i)(x+h-i)}{h (2h-2-i)(2h-1-i)(h+x)} \\ 
     & \phantom{=1+\sum\ } + 
     \frac{2(h-1) i (1 + 2h(h-2)-x)}{h (2h-3)(2h-1-i)(h+x)} \\ 
     & \phantom{=1+\sum\ } + 
     \frac{(h-2)(2h-1)(i(i-1)(x+2-h)}{h (2h-3) (2h-2-i)(2h-1-i)(h+x)} 
     \Biggr) \\ 
     & = 
     \sum_{i=0}^{h} \binom{x+h}{i} \frac{h!}{(h-i)!} \frac{(2h-1-i)!}{(2h-1)!} z^i. 
     \qedhere 
\end{align*} 
\end{proof} 

The expansions of the J-fractions, and more generally, for any continued fraction 
whose convergents are defined by the ratio of terms defined recursively as in 
\eqref{eqn_ConvFns_BinomCoeff_xpnn}, provide additional 
recurrence relations and exact finite sums for the $h^{th}$ convergent 
functions, $\Conv_{1,h}(x, z)$, given by \citep[\S 1.12]{NISTHB} 
\begin{align*} 
\Conv_{1,h}(x, z) & = \Conv_{1,h-1}(x, z) + 
     \frac{(-1)^{h-1} \ab_{x,2} \ab_{x,3} \cdots \ab_{x,h} z^{2h-2}}{ 
     \ConvQ_{1,h}(x, z) \ConvQ_{1,h-1}(x, z)} \\ 
     & = 
     \frac{1}{1+(x+1)z} + \sum_{i=2}^{h} 
     \frac{(-1)^{i-1}\binom{x+i-1}{i-1} \binom{x}{i-1} 
     \binom{2i-3}{i-2}^{-2} z^{2i-2}}{ 
     2 \cdot \ConvQ_{1,i}(x, z) \ConvQ_{1,i-1}(x, z)}. 
\end{align*} 

\begin{theorem}[Main Theorem I] 
\label{theorem_MainThmI} 
For integers $h \geq 2$, we have that the $h^{th}$ convergent functions, 
$\Conv_{1,h}(x, z)$, exactly generate the binomial coefficients $\binom{x+n}{n}$ 
for all $0 \leq n \leq h$ as $[z^n] \Conv_{1,h}(x, -z) = \binom{x+n}{n}$. 
\end{theorem} 
\begin{proof} 
Since $\Conv_{1,h}(x, z)$ is rational in $z$ for all $h \geq 2$, 
Proposition \ref{prop_NumConvFn_Formula} and 
Proposition \ref{prop_DenomConvFn_Formula} imply the following 
finite difference equations for the coefficients of the convergent functions, 
$\Conv_{1,h}(x, z)$, when $n \geq 0$ \citep[\S 2.3]{GFLECT}: 
\begin{align} 
\label{eqn_ProofOfThmI_ConvFn_FiniteDiffEqn} 
[z^n] \Conv_{1,h}(x, -z) & = -\sum_{i=1}^{\min(n, h)} 
     \binom{x+h}{i} [z^{n-i}] \Conv_{1,h}(x, z) \times 
     \frac{h!}{(h-i)!} \frac{(2h-1-i)!}{(2h-1)!} (-1)^i \\ 
\notag 
     & \phantom{=\ } + 
     [z^n] \ConvP_{1,h}(x, -z) \Iverson{0 \leq n < h}. 
\end{align} 
We must show that $[z^n] \Conv_{1,h}(x, -z) = \binom{x+n}{n}$ in the 
separate cases where $0 \leq n < h$ and when $n \equiv h$. 
For the first case, we use induction on $n$ to show our result. 
Since $[z^0] F(z) / G(z) = F(0) / G(0)$ for any functions, $F(z)$ and $G(z)$, 
with a power series expansion in $z$ about zero, we have by the two 
propositions above that $[z^0] \Conv_{1,h}(x, z) = \binom{x+0}{0} \equiv 1$ for all 
$h \geq 2$. 

Next, we suppose that for $h > n \geq 1$ and all $k < n$, 
$[z^k] \Conv_{1,h}(x, -z) = \binom{x+k}{k}$. 
Since $0 \leq n - i < n$ for all $i$ in the right-hand-side sum of 
\eqref{eqn_ProofOfThmI_ConvFn_FiniteDiffEqn}, we can apply the inductive 
hypothesis, combined with the observation in \eqref{eqn_PropPhxz_alt_stmt} from the 
proof of Proposition \ref{prop_NumConvFn_Formula}, to the 
recurrence relation in the previous equation to obtain that when $n < h$ we have 
\begin{align*} 
[z^n] \Conv_{1,h}(x, -z) & = -\sum_{i=1}^{n} 
     \binom{x+h}{i} \binom{x+n-i}{n-i} 
     \frac{h!}{(h-i)!} \frac{(2h-1-i)!}{(2h-1)!} (-1)^i \\ 
     & \phantom{=\ } + 
     \sum_{i=0}^{n} \binom{x+h}{i} \binom{x+n-i}{n-i} 
     \frac{h!}{(h-i)!} \frac{(2h-1-i)!}{(2h-1)!} (-1)^i \\ 
     & = 
     \binom{x+n}{n}. 
\end{align*} 
To prove the claim in the first special case of 
\eqref{eqn_ProofOfThmI_ConvFn_FiniteDiffEqn} 
where $[z^n] \ConvP_{1,h}(x, -z) \equiv 0$, \ie precisely when $n \equiv h$, 
we use an alternate approach to evaluating these sums by exactly summing with 
\emph{Mathematica} as 
\begin{align*} 
[z^n] \Conv_{1,h}(x, -z) & = 
     -\sum_{i=1}^{n} 
     \binom{x+h}{i} \binom{x+n-i}{n-i} 
     \frac{h!}{(h-i)!} \frac{(2h-1-i)!}{(2h-1)!} (-1)^i \\ 
     & = 
     \binom{x+n}{n} \left(1 - _3F_2(-h, -n, -(x+h); 1-2h, -(x+n); 1), 
     \right) 
\end{align*} 
where $_pF_q(a_1, \ldots, a_p; b_1, \ldots, b_q; z)$ denotes the 
\emph{generalized hypergeometric function} whose coefficients over powers of $z^k$ 
are given by the \emph{Pochhammer symbol} products, 
$\Pochhammer{a_1}{k} \cdots \Pochhammer{a_p}{k} / k! \cdot 
 \Pochhammer{b_1}{k} \cdots \Pochhammer{b_q}{k}$ \citep[\S 16]{NISTHB}. 
When $h \equiv n$, the terms contributed by the generalized hypergeometric 
function in the previous equation are zero-valued\footnote{ 
     \ie since $(-n)_{n+1+k} = 0$ for all $k \geq 0$ and where for all integers 
     $n \geq 1$ the next hypergeometric sum is evaluated exactly with 
     \emph{Mathematica} as 
     \begin{align*} 
     \sum_{k=0}^n \binom{n}{k}^2 \times \frac{(-1)^k k! \cdot (2n-1-k)!}{(2n-1)!} 
          = 0. 
     \end{align*} 
}. 
Therefore when $n \equiv h$, we have that $[z^n] \Conv_{1,h}(x, z) = \binom{x+n}{n}$. 
\end{proof} 

We can actually prove a stronger statement of Theorem \ref{theorem_MainThmI} which 
provides that $[z^n] \Conv_{1,h}(x, z) = \binom{x+n}{n}$ for all 
$0 \leq n < 2h$, though for the purposes of showing that our J-fraction 
expansions defined by Definition \ref{def_ComponentSeqs_ConvFns_BinomCoeff_xpnn} 
are correct, the proof of the theorem given above suffices to show the result. 
One consequence of the theorem is that we have the following finite sums 
exactly generating, $\binom{x+n}{n}$, for any $x$ and all $n \geq 0$, 
\ie in the case of \eqref{eqn_ProofOfThmI_ConvFn_FiniteDiffEqn} where $h \equiv n$: 
\begin{align*} 
\binom{x+n}{n} & = \sum_{i=1}^{n} \binom{x+n}{i} \binom{x+n-i}{n-i} 
     \binom{n}{i} \binom{2n-1}{i}^{-1} (-1)^{i+1} + 
     \Iverson{n = 0}. 
\end{align*} 

\begin{cor}[Congruences for $\binom{x+n}{n}$ Modulo Integers $h \geq 2$] 
\label{cor_Congruences_BinomCoeff_xpnn} 
For $h \geq 2$, let $\lambda_h(x) := \prod_{i=2}^{h} \ab_{x,i}$, or equivalently, let 
$\lambda_h(x) \equiv 
 \frac{(-1)^{h-1}}{2} \binom{x+h-1}{h-1} \binom{x}{h-1} / \binom{2h-3}{h-2}^2$. 
For all $h \geq 2$, $m \geq h$, $0 \leq x \leq h$, and $n \geq 0$, 
whenever $\lambda_m(x) \in \mathbb{Z}$ and $h \mid \lambda_m(x)$, we have the 
following congruences congruences for the binomial coefficients modulo $h$: 
\begin{align*} 
\binom{x+n}{n} & \equiv \sum_{i=1}^{n} \binom{x+m}{i} \binom{x+n-i}{n-i} \cdot 
     \frac{m!}{(m-i)!} \cdot \frac{(2m-1-i)!}{(2m-1)!} (-1)^{i+1} \\ 
     & \phantom{\equiv\ } + 
     \binom{x+n-m}{n} \binom{m-1}{n} \binom{2m-1}{n}^{-1} 
     \Iverson{0 \leq n < m} && \pmod{h}. 
\end{align*} 
\end{cor} 
\begin{proof} 
The result is an immediate corollary of 
\eqref{eqn_ProofOfThmI_ConvFn_FiniteDiffEqn} from the proof of 
Theorem \ref{theorem_MainThmI} and the 
J-fraction coefficient congruence properties cited in property (C) of 
Section \ref{subSection_Intro_J-fraction_exps} in the introduction 
\citep{FLAJOLET82} \citep[\cf \S 5.7]{GFLECT}. 
\end{proof} 

\begin{remark}[Exact Congruences for the Binomial Coefficients] 
We conjecture that in fact for all $h \geq 2$, $h > n \geq 0$, and $x < h$, 
which typically corresponds to the cases of the binomial coefficients that 
we actually wish to evaluate modulo $h$, that 
\begin{align*} 
\binom{x+n}{n} & \equiv \sum_{i=1}^{h} \binom{x+h}{i} \binom{x+n-i}{n-i} \cdot 
     \frac{h!}{(h-i)!} \cdot \frac{(2h-1-i)!}{(2h-1)!} (-1)^{i+1} && 
     \pmod{h}. 
\end{align*} 
Using this result, we can expand these special case congruences for 
$\binom{x+n}{n}$ modulo $2$, $3$, $4$, and $5$ as follows: 
\begin{align*} 
\binom{x+n}{n} & \equiv \frac{2(x+2)}{3} \binom{x+n-1}{n-1} - 
     \frac{(x+1)(x+2)}{6} \binom{x+n-2}{n-2} && \pmod{2} \\ 
\binom{x+n}{n} & \equiv \frac{3(x+3)}{5} \binom{x+n-1}{n-1} - 
     \frac{3(x+2)(x+3)}{20} \binom{x+n-2}{n-2} \\ 
     & \phantom{\equiv\ } + 
     \frac{(x+1)(x+2)(x+3)}{60} \binom{x+n-3}{n-3} && \pmod{3} \\ 
\binom{x+n}{n} & \equiv \frac{4(x+4)}{7} \binom{x+n-1}{n-1} - 
     \frac{(x+3)(x+4)}{7} \binom{x+n-2}{n-2} \\ 
     & \phantom{\equiv\ } + 
     \frac{2(x+2)(x+3)(x+4)}{105} \binom{x+n-3}{n-3} \\ 
     & \phantom{\equiv\ } - 
     \frac{(x+1)(x+2)(x+3)(x+4)}{840} \binom{x+n-4}{n-4} && \pmod{4} \\ 
\binom{x+n}{n} & \equiv \frac{5(x+5)}{9} \binom{x+n-1}{n-1} - 
     \frac{5(x+4)(x+5)}{56} \binom{x+n-2}{n-2} \\ 
     & \phantom{\equiv\ } + 
     \frac{5(x+3)(x+4)(x+5)}{252} \binom{x+n-3}{n-3} \\ 
     & \phantom{\equiv\ } - 
     \frac{5(x+2)(x+3)(x+4)(x+5)}{3024} \binom{x+n-4}{n-4} \\ 
     & \phantom{\equiv\ } + 
     \frac{(x+1)(x+2)(x+3)(x+4)(x+5)}{15120} \binom{x+n-5}{n-5} && \pmod{5}. 
\end{align*} 
\end{remark} 

\begin{cor}[Reduction of Order of the Binomial Coefficients] 
\label{cor_BinCoeffV1_Addition-ReductionFormulas} 
For integers $r, p$ such that $p \geq r \geq 0$, let the sequences, 
$k_{r,p}(x)$, be defined as 
\begin{align*} 
k_{r,p}(x) & = (-1)^{p-r} \binom{x+p}{p-r} \binom{2r}{r} \binom{p+r}{r}^{-1}. 
\end{align*} 
For all integers $p, q \geq 0$, we have an addition theorem for the 
binomial coefficients given by 
\begin{align*} 
(-1)^{p+q} \binom{x+p+q}{p+q} & = k_{0,p}(x) k_{0,q}(x) + 
     \sum_{i=1}^{\max(p, q)} \lambda_{i+1}(x) k_{i,p}(x) k_{i, q}(x), 
\end{align*} 
where $\lambda_h(x)$ is defined as in 
Corollary \ref{cor_Congruences_BinomCoeff_xpnn}. 
\end{cor} 
\begin{proof} 
We can prove the result using the matrix method from the references 
\citep[\S 1, p.\ 133]{FLAJOLET80B} \citep[\S 11]{WALL-CFRACS}. 
In particular, for integers $p \geq r \geq 0$, let the functions, 
$\widetilde{k}_{r,p}(x)$, denote the solutions to the matrix equation 
\begin{align} 
\label{eqn_krpx_matrix_eqn} 
\scriptsize\begin{bmatrix} 
\widetilde{k}_{0,1}(x) & \widetilde{k}_{1,1}(x) & 0 & 0 & \cdots \\ 
\widetilde{k}_{0,2}(x) & \widetilde{k}_{1,2}(x) & \widetilde{k}_{2,2}(x) & 0 
     & \cdots \\ 
\widetilde{k}_{0,3}(x) & \widetilde{k}_{1,3}(x) & \widetilde{k}_{2,3}(x) & 
     \widetilde{k}_{3,3}(x) & \cdots \\ 
 & \cdots 
\end{bmatrix} & = 
\scriptsize\begin{bmatrix} 
\widetilde{k}_{0,0}(x) & 0 & 0 & 0 & \cdots \\ 
\widetilde{k}_{0,1}(x) & \widetilde{k}_{1,1}(x) & 0 & 0 & \cdots \\ 
\widetilde{k}_{0,1}(x) & \widetilde{k}_{1,1}(x) & \widetilde{k}_{2,2}(x) & 0 & 
     \cdots \\ 
 & \cdots 
\end{bmatrix} \cdot 
\begin{bmatrix} 
c_{x,1} & 1 & 0 & 0 & \cdots \\ 
\ab_{x,2} & c_{x,2} & 1 & 0 & \cdots \\ 
0 & \ab_{x,3} & c_{x,3} & 1 & \cdots \\ 
  & \cdots 
\end{bmatrix}, 
\end{align} 
where we define $\widetilde{k}_{0,p}(x) := \binom{x+p}{p}$ for all $p \geq 0$. 
We claim that for all integers $p, r \geq 0$ with $p \geq r \geq 0$, the two 
functions, $k_{r,p}(x)$ and $\widetilde{k}_{r,p}(x)$, are equal. 
To prove the claim, we notice that for fixed $p$, equation 
\eqref{eqn_krpx_matrix_eqn} implies recurrence relations over $r$ given by 
\begin{align*} 
\widetilde{k}_{1,p}(x) & = \frac{1}{\ab_{x,2}}\left( 
     \widetilde{k}_{0,p+1}(x) - c_{x,1} \cdot \widetilde{k}_{0,p}(x) 
     \right),\ p \geq 1 \\ 
\widetilde{k}_{r+1,p}(x) & = \frac{1}{\ab_{x,r+2}}\left( 
     \widetilde{k}_{r,p+1}(x) - \widetilde{k}_{r-1,p}(x) - 
     c_{x,r+1} \cdot \widetilde{k}_{r,p}(x)
     \right),\ p > r \geq 1. 
\end{align*} 
The first recurrence relation provides that for $p \geq 1$ 
\begin{align*} 
\widetilde{k}_{1,p} & = -\frac{2}{x(x+1)} \left( 
     (-1)^{p+1} \binom{x+p+1}{p+1} + (-1)^{p} (x+1) \binom{x+p}{p} 
     \right) \\ 
     & = 
     \frac{2 (-1)^{p-1} (p+x)!}{(p+1) (x+1)! (p-1)!} \left( 
     \frac{(x+1)(p+1)}{p x} - \frac{(x+p+1)}{p x} 
     \right) \\ 
     & = 
     \frac{2 (-1)^{p-1}}{(p+1)} \binom{x+p}{p-1}, 
\end{align*} 
which is the same as the formula for $k_{1,p}(x)$ stated above. 
We can then use the second recurrence relation to complete the proof of the 
claim by induction on $r$. 
When $r = 0, 1$, we have that the two formulas for $k_{r,p}(x)$ and 
$\widetilde{k}_{r,p}(x)$ coincide. We suppose that the claim is true for some 
$r \geq 1$, which by the previous recurrence relation in turn implies that 
\begin{align*} 
\widetilde{k}_{r+1,p}(x) & = 
     \frac{-4 \cdot (2r+1)^2}{(x-r)(x+r+1)} \Biggl( 
     (-1)^{p-1-r} \binom{x+p+1}{p+1-r} \binom{2p+2}{p+1} \binom{r+p+1}{p+1}^{-1} \\ 
     & \phantom{= \frac{-4 \cdot (2r+1)^2}{(x-r)(x+r+1)} \Biggl(\ } - 
     (-1)^{r-1-p} \binom{x+p}{p+1-r} \binom{2r-2}{r-1} 
     \binom{p+r-1}{r-1}^{-1} \\ 
     & \phantom{= \frac{-4 \cdot (2r+1)^2}{(x-r)(x+r+1)} \Biggl(\ } + 
     \frac{(-1)^{p-1-r} (1+2(r-1)(r+1)-x)}{(2r+1)(2r-1)} 
     \binom{x+p}{p-r} \binom{2p}{p} \binom{r+p}{p}^{-1} 
     \Biggr) \\ 
     & = 
     (-1)^{p-r-1} \binom{x+p}{p-1-r} \binom{2r+2}{r+1} \binom{p+r+1}{r+1}^{-1} \times 
     \frac{4 (2r+1)^2}{(r-x)(x+r+1)} \times \Biggl( \\ 
     & \phantom{=(-1)^{p-r-1}\ } 
     \frac{(p+1)(p+1+x)(r+1+x)}{2 (p-r)(p+1-r)(2r-1)(2r+1)} - 
     \frac{(p+r)(p+1+r)(r+x)(r+1+x)}{4 (p-r)(p+1-r)(2r-1)(2r+1)} \\ 
     & \phantom{=(-1)^{p-r-1}\ } - 
     \frac{(1+2(r-1)(r+1)-x)(p+1+r)(r+1+x)}{2 (p-r)(2r-1) (2r+1)^2} 
     \Biggr) \\ 
     & = 
     (-1)^{p-r-1} \binom{x+p}{p-1-r} \binom{2r+2}{r+1} \binom{p+r+1}{r+1}^{-1}. 
\end{align*} 
Since the two formulas are equivalent, we obtain the next form of the 
\emph{addition formula}, or alternately, a formula providing a 
``\emph{reduction}'' of the order of the upper and lower indices to the 
binomial coefficients, 
$\binom{x+n}{n}$, given by the expansion in the references in the 
following forms for all integers $p, q \geq 0$: 
\begin{align*} 
(-1)^{p+q} \binom{x+p+q}{p+q} & = 
     k_{0,p} k_{0,q} + \ab_{x,2} \cdot k_{1,p} k_{1,q} + 
     \ab_{x,2} \ab_{x,3} \cdot k_{2,p} k_{2,q} + \cdots \\ 
     & = 
     (-1)^{p+q} \binom{x+p}{p} \binom{x+q}{q} \\ 
     & \phantom{=\ } + 
     \sum_{i=1}^{\max(p, q)} 
     (-1)^{p+q} \lambda_{i+1}(x) \binom{x+p}{p-i} \binom{x+q}{q-i} \binom{2i}{i}^2 
     \binom{p+i}{i}^{-1} \binom{q+i}{i}^{-1}. 
     \qedhere 
\end{align*} 
\end{proof} 

We notice that this result provides \emph{exact} finite sum expansions of the 
binomial coefficients, $\binom{x+n}{n}$, for any $x$ and $n \geq 0$ which hold 
modulo any integers $h \geq 2$ (prime or composite), and compare the 
reduction in the upper and lower coefficient indices to the 
congruence result provided by Lucas' theorem and its related variants 
stated in the introduction. 

\section{J-Fraction Expansions for the Binomial Coefficients, $\binom{x}{n}$} 
\label{Section_BinomCoeffs_xn}

We can construct (and formally prove) similar convergent-based J-fraction 
constructions enumerating the binomial coefficients, $\binom{x}{n}$, for an 
indeterminate $x$ and $n \geq 0$. Since the proofs for the 
propositions, theorem, and corollaries in this section 
are almost identical to those given in the 
case of the binomial coefficient variants, $\binom{x+n}{n}$, in 
Section \ref{Section_BinomCoeffs_xpnn}, we omit the 
proofs of the next results stated below. 
We begin with the definition of the component sequences and convergent functions 
corresponding to the J-fractions generating the binomial coefficients, 
$\binom{x}{n}$, in this case. 

\begin{definition}[Component Sequences and Convergent Functions for $\binom{x}{n}$] 
\label{def_CompSeqs_ConvFns_BinCoeff_xn_v2} 
For a fixed non-zero indeterminate $x$ and $i \geq 1$, we define the 
component sequences, $c_{x,i}^{(2)}$ and $\ab_{x,i}^{(2)}$, as follows: 
\begin{align*} 
c_{x,i}^{(2)} & := -\frac{1}{(2i-1)(2i-3)}\left(x+2 (i-1)^2\right) \\ 
\ab_{x,i}^{(2)} & := 
     \begin{cases} 
     -\frac{1}{4 (2i-3)^2} (x-i+2)(x+i-1) & \text{ if } i \geq 3 \\ 
     -\frac{1}{2}x(x+1) & \text{ if } i = 2 \\ 
     0 & \text{ otherwise. }
     \end{cases} 
\end{align*} 
For $h \geq 0$, we define the $h^{th}$ convergent functions, 
$\Conv_{2,h}(x, z) := \ConvP_{2,h}(x, z) / \ConvQ_{2,h}(x, z)$, 
recursively through the component functions in the forms of 
\begin{align} 
\label{eqn_ConvFns_BinomCoeff_xn} 
\ConvP_{2,h}(x, z) & = (1-c_{x,h}^{(2)}z) \ConvP_{2,h-1}(x, z) - 
     \ab_{x,h}^{(2)} z^2 \ConvP_{2,h-1}(x, z) + \Iverson{h = 1} \\ 
\notag 
\ConvQ_{2,h}(x, z) & = (1-c_{x,h}^{(2)}z) \ConvQ_{2,h-1}(x, z) - 
     \ab_{x,h}^{(2)} z^2 \ConvQ_{2,h-1}(x, z) + 
     (1-c_{x,1}^{(2)}z) \Iverson{h = 1} + \Iverson{h = 0}. 
\end{align} 
Listings of the first several cases of the numerator and denominator 
convergent functions, $\ConvP_{2,h}(x, z)$ and $\ConvQ_{2,h}(x, z)$, 
are given in Table \ref{table_ConvP2v2hxz_SpCases} and 
Table \ref{table_ConvQ2v2hxz_SpCases}, respectively. 
\end{definition} 

\begin{table}[ht] 
\begin{framed} 
\centering\small

\begin{tabular}{||c|l||} \hline\hline 
$h$ & $\ConvP_{2,h}(x, z)$ \\ \hline
0 & $0$ \\ 
1 & $1$ \\
2 & $1 + \frac{x+2}{3}z$ \\
3 & $1+\frac{2}{5} (x+3) z+\frac{1}{20} (x+2) (x+3) z^2$ \\
4 & $1+\frac{3}{7} (x+4) z+\frac{1}{14} (x+3) (x+4) z^2+ 
     \frac{1}{210} (x+2) (x+3) (x+4) z^3$ \\
5 & $1+\frac{4}{9} (x+5) z+\frac{1}{12} (x+4) (x+5) z^2+ 
     \frac{1}{126} (x+3) (x+4) (x+5) z^3+ 
     \frac{(x+2) (x+3) (x+4) (x+5)}{3024} z^4$ \\
6 & $1+\frac{5}{11} (x+6) z+\frac{1}{11} (x+5) (x+6) z^2+ 
     \frac{1}{99} (x+4) (x+5) (x+6) z^3+ 
     \frac{(x+3) (x+4) (x+5) (x+6)}{1584} z^4$ \\ 
  & $\phantom{1} + 
     \frac{(x+2) (x+3) (x+4) (x+5) (x+6)}{55440} z^5$ \\
\hline\hline 
\end{tabular} 

\bigskip 

\captionof{table}{The Numerator Convergent Functions, $\ConvP_{2,h}(x, z)$, 
                  Generating the Binomial Coefficients, $\binom{x}{n}$} 
\label{table_ConvP2v2hxz_SpCases} 

\end{framed} 
\end{table} 

\begin{table}[ht] 
\begin{framed} 
\centering\small

\begin{tabular}{||c|l||} \hline\hline 
$h$ & $\ConvQ_{2,h}(x, z)$ \\ \hline
1 & $1-xz$ \\
2 & $1-\frac{2}{3} (x-1) z+\frac{1}{6} (x-1) x z^2$ \\
3 & $1-\frac{3}{5} (x-2) z+\frac{3}{20} (x-2) (x-1) z^2- 
     \frac{1}{60} (x-2) (x-1) x z^3$ \\
4 & $1-\frac{4}{7} (x-3) z+\frac{1}{7} (x-3) (x-2) z^2- 
     \frac{2}{105} (x-3) (x-2) (x-1) z^3+ 
     \frac{1}{840} (x-3) (x-2) (x-1) x z^4$ \\
5 & $1-\frac{5}{9} (x-4) z+\frac{5}{36} (x-4) (x-3) z^2- 
     \frac{5}{252} (x-4) (x-3) (x-2) z^3+ 
     \frac{5 (x-4) (x-3) (x-2) (x-1)}{3024} z^4$ \\ 
  & $\phantom{1} - 
     \frac{(x-4) (x-3) (x-2) (x-1) x}{15120} z^5$ \\
6 & $1-\frac{6}{11} (x-5) z+\frac{3}{22} (x-5) (x-4) z^2- 
     \frac{2}{99} (x-5) (x-4) (x-3) z^3$ \\ 
  & $\phantom{1} + 
     \frac{1}{528} (x-5) (x-4) (x-3) (x-2) z^4- 
     \frac{(x-5) (x-4) (x-3) (x-2) (x-1)}{9240} z^5$ \\ 
  & $\phantom{1} + 
     \frac{(x-5) (x-4) (x-3) (x-2) (x-1) x}{332640} z^6$ \\
\hline\hline 
\end{tabular} 

\bigskip 

\captionof{table}{The Denominator Convergent Functions, $\ConvQ_{2,h}(x, z)$, 
                  Generating the Binomial Coefficients, $\binom{x}{n}$} 
\label{table_ConvQ2v2hxz_SpCases} 

\end{framed} 
\end{table} 

\begin{prop}[Convergent Function Formulas] 
For all $h \geq 2$ and a fixed indeterminate $x$, the numerator and denominator 
convergent functions, $\ConvP_{2,h}(x, z)$ and $\ConvQ_{2,h}(x, z)$, each 
satisfy the following respective formulas: 
\begin{align*} 
\ConvP_{2,h}(x, z) & = \sum_{n=0}^{h-1} \binom{x+h}{n} 
     \frac{(h-1)!}{(h-1-n)!} \frac{(2h-1-n)!}{(2h-1)!} z^n \\ 
     & = 
     \sum_{n=0}^{h-1} \binom{x+h}{n} \binom{h-1}{n} \binom{2h-1}{n}^{-1} z^n \\ 
\ConvQ_{2,h}(x, z) & = 
     \sum_{i=0}^{h} \binom{x-h+i}{i} 
     \frac{h!}{(h-i)!} \frac{(2h-1-i)!}{(2h-1)!} (-z)^i. 
\end{align*} 
\end{prop} 

\begin{theorem}[Main Theorem II] 
For all integers $h \geq 2$ and $0 \leq n \leq h$, we have that the $h^{th}$ 
convergent functions, $\Conv_{2,h}(x, z)$, exactly generate the binomial 
coefficients $\binom{x}{n}$ as $[z^n] \Conv_{2,h}(x, z) = \binom{x}{n}$. 
\end{theorem} 

As in Section \ref{Section_BinomCoeffs_xpnn}, we remark that 
one consequence of the second main theorem 
resulting from the proposition above is that we have a new proof of the next 
exact formula generating the binomial coefficients, $\binom{x}{n}$, for any 
$x$ and all $n \geq 0$. 
\begin{align*} 
\binom{x}{n} & = \sum_{i=1}^{n} \binom{x-n+i}{i} \binom{x}{n-i} 
     \binom{n}{i} \binom{2n-1}{i}^{-1} (-1)^{i+1} + 
     \Iverson{n = 0} 
\end{align*} 

\begin{cor}[Congruences for $\binom{x}{n}$ Modulo Integers $h \geq 2$] 
Let the $h^{th}$ modulus, $\lambda_h(x)$, of the J-fraction defined by 
Definition \ref{def_CompSeqs_ConvFns_BinCoeff_xn_v2} correspond to the 
definitions given in Corollary \ref{cor_Congruences_BinomCoeff_xpnn}. 
For all $h \geq 2$, $m \geq h$, $0 \leq x \leq h$, and $n \geq 0$, 
whenever $\lambda_m(x) \in \mathbb{Z}$ and $h \mid \lambda_m(x)$, we have the 
following congruences congruences for the binomial coefficients, 
$\binom{x}{n}$, modulo $h$: 
\begin{align*} 
\binom{x}{n} & \equiv \sum_{i=1}^{n} \binom{x-m+i}{i} \binom{x}{n-i} \cdot 
     \frac{m!}{(m-i)!} \cdot \frac{(2m-1-i)!}{(2m-1)!} (-1)^{i+1} \\ 
     & \phantom{\equiv\ } + 
     \binom{x+m}{n} \binom{m-1}{n} \binom{2m-1}{n}^{-1} 
     \Iverson{0 \leq n < m} && \pmod{h}. 
\end{align*} 

\end{cor} 

\begin{remark}[Exact Congruences for Special Cases] 
We again conjecture that for all $h \geq 2$ and all $x, n \geq 0$, we have 
congruences for the binomial coefficients, $\binom{x}{n}$, of the form 
\begin{align*} 
\binom{x}{n} & \equiv \sum_{i=1}^{h} \binom{x-h+i}{i} \binom{x}{n-i} 
     \binom{h}{i} \binom{2h-1}{i}^{-1} (-1)^{i+1} \\ 
     & \phantom{\equiv\sum\ } + 
     \binom{x+h}{n} \binom{h-1}{n} \binom{2h-1}{n}^{-1} && \pmod{h}. 
\end{align*} 
Using this result, we can then expand the first several special cases of these 
congruences as follows: 
\begin{align*} 
\binom{x}{n} & \equiv 
     \frac{2(x-1)}{3} \binom{x}{n-1} - \frac{x(x-1)}{6} \binom{x}{n-2} + 
     \frac{(n-2)(n-3)}{6} \binom{x+2}{n} \Iverson{n \leq 1} && \pmod{2} \\ 
\binom{x}{n} & \equiv 
     \frac{3(x-2)}{5} \binom{x}{n-1} - \frac{3 (x-1)(x-2)}{20} \binom{x}{n-2} + 
     \frac{x(x-1)(x-2)}{60} \binom{x}{n-3} \\ 
     & \phantom{\equiv \frac{3(x-2)}{5} \binom{x}{n-1}\ } - 
     \frac{(n-3)(n-4)(n-5)}{60} \binom{x+3}{n} \Iverson{n \leq 2} && \pmod{3} \\ 
\binom{x}{n} & \equiv 
     \frac{4(x-3)}{7} \binom{x}{n-1} - \frac{(x-2)(x-3)}{7} \binom{x}{n-2} + 
     \frac{2 (x-1)(x-2)(x-3)}{105} \binom{x}{n-3} \\ 
     & \phantom{\equiv \frac{3(x-2)}{5} \binom{x}{n-1}\ } - 
     \frac{x(x-1)(x-2)(x-3)}{840} \binom{x}{n-4} \\ 
     & \phantom{\equiv \frac{3(x-2)}{5} \binom{x}{n-1}\ } + 
     \frac{(n-4)(n-5)(n-6)(n-7)}{840} \binom{x+4}{n} \Iverson{n \leq 3} && \pmod{4} \\ 
\binom{x}{n} & \equiv 
     \frac{5(x-4)}{9} \binom{x}{n-1} - \frac{5(x-3)(x-4)}{36} \binom{x}{n-2} + 
     \frac{5 (x-2)(x-3)(x-4)}{252} \binom{x}{n-3} \\ 
     & \phantom{\equiv \frac{3(x-2)}{5} \binom{x}{n-1}\ } - 
     \frac{5(x-1)(x-2)(x-3)(x-4)}{3024} \binom{x}{n-4} \\ 
     & \phantom{\equiv \frac{3(x-2)}{5} \binom{x}{n-1}\ } - 
     \frac{x(x-1)(x-2)(x-3)(x-4)}{15120} \binom{x}{n-5} \\ 
     & \phantom{\equiv \frac{3(x-2)}{5} \binom{x}{n-1}\ } - 
     \frac{(n-5)(n-6)(n-7)(n-8)(n-9)}{15120} \binom{x+5}{n} 
     \Iverson{n \leq 4} && \pmod{5}. 
\end{align*} 
\end{remark} 

\begin{cor}[Reduction Formulas for $\binom{x}{n}$] 
\label{cor_BinCoeffV2_Addition-ReductionFormulas} 
Let the functions, $k_{r,s}(x)$ be defined for all $s \geq r \geq 0$ as 
\begin{align*} 
k_{r,s}(x) & = \binom{x-r}{s-r} \binom{2r}{r} \binom{r+s}{r}^{-1}. 
\end{align*} 
We have a corresponding addition, or lower index reduction, formula for the binomial 
coefficients, $\binom{x}{p+q}$, given by the following equations for 
integers $p, q \geq 0$: 
\begin{align*} 
\binom{x}{p+q} & = k_{0,p}(x) k_{0,q}(x) + \sum_{i=1}^{\max(p, q)} 
     \lambda_{i+1}(x) k_{i,p}(x) k_{i,q}(x) \\ 
     & = 
     \binom{x}{p} \binom{x}{q} + \sum_{i=1}^{\max(p, q)} 
     \frac{(-1)^i \binom{x-i}{p-i} \binom{x-i}{q-i} \binom{x+i}{i} 
     \binom{x}{i} \binom{2i}{i}^2}{ 
     2 \cdot \binom{p+i}{i} \binom{q+i}{i} \binom{2i-1}{i-1}^{2}}. 
\end{align*} 
\end{cor} 
We can again compare the statements of the two summation formulas in the 
corollary to the binomial coefficient expansions in the statement of 
Lucas' theorem from the introduction modulo any prime $p$. 

\section{Conclusions} 

We have established the forms of two new Jacobi-type J-fraction expansions 
providing ordinary generating functions for the binomial coefficients, 
$\binom{x+n}{n}$ and $\binom{x}{n}$, for any $x$ and all $n \geq 0$. 
The key ingredients to the proofs of these series expansions are the 
rationality of the $h^{th}$ convergents, $\Conv_{i,h}(x, z)$, for all $h \geq 1$ and 
closed-form formulas for the corresponding numerator and denominator function 
sequences, which are finite-degree polynomials in $z$ with 
$\deg_z \{\ConvP_{i,h}(x, z)\} = h-1$ and $\deg_z \{\ConvQ_{i,h}(x, z)\} = h$ for 
all $h \geq 1$ when $i = 1, 2$. 
The finite difference equations implied by the rationality of the $h^{th}$ 
convergent functions at each $h$ lead to new exact formulas for, and new 
congruences satisfied by, the two binomial coefficient variants studied within the 
article. 
We note that unlike the J-fraction expansions enumerating the generalized 
factorial functions studied in the reference \citep{MULTIFACT-CFRACS}, the 
$h^{th}$ modulus, $\lambda_h(x)$, of these J-fractions defined in 
Corollary \ref{cor_Congruences_BinomCoeff_xpnn} is not strictly integer-valued 
for all $x$ and $h$, which complicates the formulations of the 
new congruence properties for the binomial coefficient variants considered in the 
article. 

We also compare the two forms of the addition, or \emph{reduction} of index, 
formulas for these coefficients stated in 
Corollary \ref{cor_BinCoeffV1_Addition-ReductionFormulas} and in 
Corollary \ref{cor_BinCoeffV2_Addition-ReductionFormulas} to the forms of 
Lucas's theorem and to the prime power congruences cited in the introduction, 
which similarly reduce the upper and lower indices to these sequences 
with respect to powers of prime moduli. 
New research directions based on the results in this article include 
further study of the binomial coefficient congruences given in 
Section \ref{Section_BinomCoeffs_xpnn} and 
Section \ref{Section_BinomCoeffs_xn} modulo composite integers $h \geq 4$. 
In particular, one direction for future research based on these topics is to 
find exact formulas for doubly-indexed sequences of multipliers, 
$\widetilde{m}_{i,h,n}$, such that 
$\binom{x+n}{n} \equiv \widetilde{m}_{1,h,n} [z^n] \Conv_{1,h}(x, z) \pmod{h}$ 
and 
$\binom{x}{n} \equiv \widetilde{m}_{2,h,n} [z^n] \Conv_{2,h}(x, z) \pmod{h}$ 
for all integers $x, n \geq 0$. 
This application is surely a non-trivial task worthy of further study and 
consideration based on the results we prove here.  

% ////// (THE BIBLIOGRAPHY) ///////////////////////////////////////////////////

%\vskip 0.1in 
%\textbf{\hrule}\bigskip 
%\vskip 0.1in 

\renewcommand{\refname}{References} 
%\bibliographystyle{abbrv-mod} 
%\bibliography{ijnt-working-draft} 

\end{document}